\documentclass[final,leqno]{siamltex}

\usepackage{amsfonts,amssymb} %%%add
\usepackage{amsmath} %%%add
\usepackage{color} %%%add
\usepackage{graphicx} %%%add
\usepackage{epsfig,mathrsfs} %%%add
\usepackage[bf,SL,BF]{subfigure} %%%add
\usepackage{fancyhdr} %%%add
\usepackage{CJK} %%%add
\usepackage{caption} %%%add
\usepackage{wrapfig} %%%add
\usepackage{cases} %%%add
\usepackage{bm}%%%add
\newtheorem{remark}{Remark}[section] %%%add
\newtheorem{example}{Example}[section] %%%add

\title{Discretized fractional substantial calculus  \thanks{
This work was supported by the National Natural Science Foundation of China under
Grant  No. 11271173.}}

\author{Minghua Chen \thanks{ School of Mathematics and Statistics,
Lanzhou University, Lanzhou 730000, P. R. China. }
        \and Weihua Deng \thanks{Corresponding author (Email: dengwh@lzu.edu.cn). School of Mathematics and Statistics,
Lanzhou University, Lanzhou 730000, P. R. China.}}

\begin{document}

\maketitle

\begin{abstract}
This paper discusses the properties and the numerical discretizations of the fractional substantial integral
$$I_s^\nu f(x)=\frac{1}{\Gamma(\nu)} \int_{a}^x{\left(x-\tau\right)^{\nu-1}}e^{-\sigma(x-\tau)}{f(\tau)}d\tau,~~\nu>0, $$
and the fractional substantial derivative
$$D_s^\mu f(x)=D_s^m[I_s^\nu f(x)],~~~~\nu=m-\mu,$$
where $D_s=\frac{\partial}{\partial x}+\sigma=D+\sigma$, $\sigma$ can be a constant or a function without related to $x$, say $\sigma(y)$; and $m$ is the smallest integer that exceeds $\mu$. The Fourier transform method and fractional linear multistep method are used to analyze the properties or derive the discretized schemes. And the convergences of the presented discretized schemes with the global truncation error $\mathcal{O}(h^p)$ $ (p=1,2,3,4,5)$ are theoretically proved and numerically verified.
\end{abstract}



\pagestyle{myheadings}
\thispagestyle{plain}
\markboth{M. H. CHEN  AND W. H.  DENG }{DISCRETIZED FRACTIONAL SUBSTANTIAL CALCULUS}

\section{Introduction}
Anomalous diffusion processes are usually characterized by the nonlinear time dependance of the mean squared displacement, i.e., $\langle z^2(t)\rangle \sim t^\alpha$. When $0<\alpha<1$, it is called subdiffusion; $1<\alpha$ corresponds to superdiffusion, and $\alpha=1$ to normal diffusion.  A versatile framework for describing the anomalous diffusion is the continuous time random walks (CTRWs), which is governed by the waiting time probability density function (PDF) and jump length PDF. When the waiting time PDF and/or jump length PDF are power-law, and the two PDFs are independent, the transport equations can be derived, namely fractional Fokker-Planck and Klein-Kramers equations \cite{Metzler:00}. The time fractional Fokker-Planck equation can well characterize the subdiffusion, and the space fractional Fokker-Planck equation can depict the L\'{e}vy flight. The L\'{e}vy flight has a diverging mean squared displacement, and can just be applied to rather exotic physical processes \cite{Sokolov:03}.

L\'{e}vy walk gives another proper dynamical description for the superdiffusion (roughly speaking, now the particle has finite physical speed), and the PDFs of waiting time and jump length are spatiotemporal coupling \cite{Sokolov:03}. Friedrich and his co-workers discuss the CTRW model with position-velocity coupling PDF \cite{Friedrich:06}. Carmi and Barkai use the CTRW model with functional of path and position coupling PDF \cite{Carmi:11}. Based on the CTRW models with coupling PDFs, they all derive the deterministic equations; and mathematically an important operator, fractional substantial derivative, is introduced \cite{Carmi:11, Carmi:10, Friedrich:06, Sokolov:03, Turgeman:09}.

With the wide applications of the fractional substantial derivative, it seems to be urgent to mathematically analyze its properties and numerically provide its effective discretizations. This paper focuses on these two topics. The fractional substantial  derivative is defined by \cite{Carmi:11, Friedrich:06}
\begin{equation*}
D_s^\nu f(x)=\frac{1}{\Gamma(\nu)} \left[\frac{\partial}{\partial x}+\sigma\right]
\int_{0}^x{\left(x-\tau\right)^{\nu-1}}e^{-\sigma (x-\tau)}{f(\tau)}d\tau,~~0<\nu<1,
\end{equation*}
where $\sigma$ can be a constant or a function not related to $x$, say, $\sigma(y)$. In this paper, we extend the order of fractional substantial derivative $\nu\in (0,1)$ to $\nu>0$. First, we introduce the fractional substantial integral.
\begin{definition}\label{definition1.1}
Let $ \nu>0$, $f(x)$ be piecewise continuous on $(a,\infty)$ and integrable on any finite subinterval $[a,\infty)$; and let $\sigma$ be a constant or a function without related to $x$. Then the fractional substantial  integral of $f$ of order $\nu$ is defined as
\begin{equation}\label{1.1}
I_s^\nu f(x)=\frac{1}{\Gamma(\nu)}\int_{a}^x{\left(x-\tau\right)^{\nu-1}}e^{-\sigma(x-\tau)}{f(\tau)}d\tau, ~~~~x>a.
\end{equation}
\end{definition}

\begin{definition}\label{definition1.2}
Let $ \mu>0$,  $f(x)$ be (m-1)-times continuously differentiable on $(a,\infty)$ and its m-times derivative be integrable on any finite subinterval of  $[a,\infty)$, where $m$ is the smallest integer that exceeds $\mu$; and let $\sigma$ be a constant or a function without related to $x$. Then
the fractional substantial  derivative of $f$ of order $\mu$ is defined as
\begin{equation}\label{1.2}
D_s^\mu f(x)=D_s^m[I_s^{m-\mu} f(x)],
\end{equation}
where
\begin{equation}\label{1.3}
  D_s^m=\left(\frac{\partial}{\partial x}+\sigma\right)^m=(D+\sigma)^m=(D+\sigma)(D+\sigma)\cdots(D+\sigma).
\end{equation}
\end{definition}

When $\sigma=0$, obviously, the fractional substantial integral and derivative reduce to the Riemann-Liouville
fractional integral and derivative, respectively.

In the following, using Fourier transform methods and fractional linear multistep methods, respectively, we derive the $p$-th order ($p\leq 5$)  approximations of the $\alpha$-th fractional substantial derivative ($\alpha>0$) or
fractional substantial integral ($\alpha<0$) by the corresponding coefficients of the generating functions $\kappa ^{p,\alpha}(\zeta)$, with
\begin{equation}\label{1.4}
\kappa^{p,\alpha}(\zeta) = \left(\sum_{i=1}^p\frac{1}{i}\left(1- e^{-\sigma h} \zeta  \right)^i\right)^{\alpha},
\end{equation}
where $h$ is the uniform stepsize. We rewrite (\ref{1.4}) as a tabular, see Table \ref{table:01}.

\begin{table}[h]\fontsize{9.5pt}{12pt}\selectfont%生成浮动表格
 \begin{center}%\def\tabcolsep{28.5pt}%表格居中
  \caption {Generating functions of the coefficients for the $p$-th order approximation of $\alpha$-th fractional substantial derivative.} \vspace{5pt}% 标题，离表格一定的距离
\begin{tabular*}{\linewidth}{@{\extracolsep{\fill}}*{3}{c|l}}                                  \hline  %画顶端的横线
$~~~p$& ~~~~~~~~~~~~~~~~~~~~~~~~~~~~~~~~~~~~~~~~~~~~~~$\kappa^{p,\alpha}(\zeta)$   \\\hline
~~~1&    $\left(1- e^{-\sigma h} \zeta \right)^{\alpha}$    \\
~~~2& $\left(3/2-2e^{-\sigma h} \zeta +1/2(e^{-\sigma h} \zeta )^2\right)^{\alpha}$  \\
~~~3& $\left(11/6-3e^{-\sigma h} \zeta +3/2(e^{-\sigma h} \zeta )^2-1/3(e^{-\sigma h} \zeta )^3\right)^{\alpha}$  \\
~~~4& $\left(25/12-4e^{-\sigma h} \zeta +3(e^{-\sigma h} \zeta )^2-4/3(e^{-\sigma h} \zeta )^3+1/4(e^{-\sigma h} \zeta )^4\right)^{\alpha}$  \\
~~~5& $\left(137/60-5e^{-\sigma h} \zeta +5(e^{-\sigma h} \zeta )^2-10/3(e^{-\sigma h} \zeta )^3+5/4(e^{-\sigma h} \zeta )^4
   -1/5(e^{-\sigma h} \zeta )^5\right)^{\alpha}$  \\ \hline
    \end{tabular*}\label{table:01}%\vspace{-15pt}
  \end{center}
\end{table}

For $\sigma=0$, formula $(\ref{1.4})$
reduces to the fractional Lubich's methods \cite{Lubich:86}. For $\sigma=0, \alpha=1$, the scheme reduces to the classical $(p+1)$-point backward difference formula  \cite{Henrici:62}.

The  outline of this paper is as follows.
In Section 2, we give some properties of the fractional substantial  calculus.
In Sections 3 and 4, using  Fourier transform method and fractional linear multistep method, respectively, we derive the
 convergence of the discretized schemes of the fractional substantial calculus.
And  the convergence with the global truncation error $\mathcal{O}(h^p)$ $(p=1,2,3,4,5)$ are numerically verified in Section 5.
Finally, we conclude the paper with some remarks in the last section.

\section{Properties for the fractional substantial  calculus}

Let us now consider some properties of the fractional substantial  calculus.

\begin{lemma} \label {lemma2.1}
Let $f(x)$ be continuous on $[a,\infty)$, and $\nu >0$, then  for all $x\geq a$,
$$\lim _{\nu\rightarrow 0}I_s^\nu f(x)= f(x).$$
Hence we can put $I_s^0 f(x)= f(x).$
\end{lemma}
\begin{proof}
If $f(x)$ has continuous derivative for $x\geq a$, then using integration by parts to (\ref{1.1}), there exists
\begin{equation*}
\begin{split}
I_s^\nu f(x)&=-\frac{1}{\Gamma(\nu+1)}\int_{a}^x e^{-\sigma(x-\tau)}{f(\tau)}d\left(x-\tau\right)^{\nu} \\
            &=\frac{ \left(x-a\right)^{\nu} e^{-\sigma(x-a)}f(a)  }{\Gamma(\nu+1)}
              + \frac{1}{\Gamma(\nu+1)}\int_{a}^x  \left(x-\tau\right)^{\nu} e^{-\sigma(x-\tau)}{D_sf(\tau)}d\tau,
\end{split}
\end{equation*}
where $D_s$ is defined by (\ref{1.3}). So we get
\begin{equation*}
\begin{split}
 &\lim _{\nu\rightarrow 0}I_s^\nu f(x)\\
 &\quad = e^{-\sigma(x-a)}f(a) +\sigma \int_{a}^x   e^{-\sigma(x-\tau)}{f(\tau)}d\tau + \int_{a}^x   e^{-\sigma(x-\tau)}d{f(\tau)}=f(x).
\end{split}
\end{equation*}

If $f(x)$ is only continuous for $x\geq a$, the similar arguments can be performed as  \cite[p. 66-67]{Podlubny:99}, we omit it here.
\end{proof}

\begin{lemma} \label {lemma2.2}
Let $f(x)$ be continuous on $[a,\infty)$ and $\mu,\nu >0$, then  for all $x \geq a$,
$$I_s^\nu[I_s^\mu f(x)]=I_s^{\mu +\nu}f(x)=I_s^\mu[I_s^\nu f(x)].$$
\end{lemma}
\begin{proof}
\begin{equation*}
\begin{split}
     I_s^\nu[I_s^\mu f(x)]&=\frac{1}{\Gamma(\nu)}\int_{a}^x{\left(x-\tau\right)^{\nu-1}}e^{-\sigma(x-\tau)}{[I_s^\mu f(\tau)]}d\tau       \\
                          &=\frac{1}{\Gamma(\mu)\Gamma(\nu)}\int_{a}^x{\left(x-\tau\right)^{\nu-1}}e^{-\sigma(x-\tau)}d\tau
                            \int_{a}^\tau{\left(\tau-\xi\right)^{\mu-1}}e^{-\sigma(\tau-\xi)}f(\xi) d\xi    \\
                          &=\frac{1}{\Gamma(\mu)\Gamma(\nu)}\int_{a}^x e^{-\sigma(x-\xi)}f(\xi) d\xi
                            \int_{\xi}^x \left(x-\tau \right)^{\nu-1} \left(\tau-\xi \right)^{\mu-1} d\tau   \\
                          &=I_s^{\mu +\nu}f(x),
\end{split}
\end{equation*}
where the integral
$$\int_{\xi}^x \left(x-\tau \right)^{\nu-1} \left(\tau-\xi \right)^{\mu-1} d\tau =\frac{\Gamma(\mu)\Gamma(\nu)}{\Gamma(\mu+\nu)}(x-\xi)^{\mu+\nu-1}.$$
\end{proof}

\begin{lemma}\label {lemma2.3}
Let $f(x)$ be (m-1)-times continuously differentiable on $(a,\infty)$ and its m-times derivative be integrable on any finite subinterval of  $[a,\infty)$ and $\nu>0$, where $m$ is the smallest integer that exceeds $\nu$. Then for all  $x\geq a$,
$$D_s^\nu[I_s^\nu f(x)]=f(x).$$
\end{lemma}
\begin{proof}
Let us first consider the case of integer $\nu=m\geq 1:$
\begin{equation*}
  \begin{split}
D_s^m[I_s^m f(x)]&=D_s^m \left[\frac{1}{(m-1)!}\int_{a}^x{\left(x-\tau\right)^{m-1}}e^{-\sigma(x-\tau)}{f(\tau)}d\tau \right]\\
                 &=D_s  \int_{a}^xe^{-\sigma(x-\tau)}{f(\tau)}d\tau =D_s [I_s f(x)]=f(x).\\
  \end{split}
\end{equation*}
For $m-1 < \nu <m$, from Lemma \ref{lemma2.2}, there exists
$$I_s^m=I_s^{m-\nu}[I_s^\nu f(x)].$$
Thus, using  (\ref{1.2}) and above equation, we obtain
$$D_s^\nu[I_s^\nu f(x)]=D_s^m \{I_s^{m-\nu}[I_s^\nu f(x)] \}=D_s^m[I_s^m f(x)]=f(x).$$

\end{proof}

\begin{lemma}\label {lemma2.4}
Let $f(x)$ be (r-1)-times continuously differentiable on $(a,\infty)$ and its r-times derivative be integrable on any finite subinterval of  $[a,\infty)$, where $r=\max (m,n)$, $m$ and $n$ are positive integers. Denoting that
$$m-\nu=n-\mu, ~~\mu>0, \nu>0,$$
then  for all $x\geq a$,
$$D_s^n[I_s^\mu f(x)]=D_s^m[I_s^\nu f(x)].$$
\end{lemma}
\begin{proof}
If $m=n$, the lemma is trivial. Supposing that $n>m$ and $\gamma=n-m>0$, it yields
$\mu=\nu+\gamma>0$. Then according to Lemmas \ref {lemma2.2} and \ref {lemma2.3}, we obtain
$$D_s^\gamma[I_s^{\nu+\gamma} f(x)]=D_s^\gamma[I_s^{\gamma} I_s^\nu f(x)]=I_s^\nu f(x).$$
Letting $D_s^{m}$ perform on both sides of the above equation leads to
$$D_s^{m+\gamma}[I_s^{\nu+\gamma} f(x)]=D_s^m[I_s^\nu f(x)],$$
that is
$$D_s^n[I_s^\mu f(x)]=D_s^m[I_s^\nu f(x)].$$

\end{proof}

\begin{lemma} \label {lemma2.5}
Let $f(x)$ be continuously differentiable on $[a,\infty)$, and $\nu >0$. Then  for all $x\geq a$,
\begin{equation}\label{2.1}
 I_s^{\nu+1}[D_sf(x)]=I_s^\nu f(x)-\frac{f(a)}{\Gamma(\nu+1)}(x-a)^\nu e^{-\sigma (x-a)};
 \end{equation}
 and
 \begin{equation}\label{2.2}
 D_s[I_s^\nu f(x)]=I_s^\nu [D_s f(x)]+\frac{f(a)}{\Gamma(\nu)}(x-a)^{\nu-1} e^{-\sigma (x-a)}.
\end{equation}
\end{lemma}

\begin{proof}
Using integration by parts, it is easy to get
$$I_s^\nu f(x)=\frac{f(a)}{\Gamma(\nu+1)}(x-a)^\nu e^{-\sigma (x-a)}
+\frac{1}{\Gamma(\nu+1)}\int_{a}^x{\left(x-\tau\right)^{\nu}}e^{-\sigma(x-\tau)}[D_s{f(\tau)}]d\tau$$
where $D_s$ is defined by (\ref{1.3}).  Thus we obtain (\ref{2.1}).

Next we prove (\ref{2.2}). From (\ref{2.1}), it leads to
 \begin{equation*}
 \begin{split}
  &D_s[I_s^\nu f(x)]\\
  &\quad =D_s \left \{ I_s^{\nu+1}[D_sf(x)]+\frac{f(a)}{\Gamma(\nu+1)}(x-a)^\nu e^{-\sigma (x-a)}\right \}\\
%  &\quad=D (I_s^{\nu+1}[D_sf(x)]) +\sigma (I_s^{\nu+1}[D_sf(x)]) +\frac{f(a)}{\Gamma(\nu+1)} (D+\sigma)\left[(x-a)^\nu e^{-\sigma (x-a)}\right]\\
  &\quad=I_s^\nu [D_s f(x)]+\frac{f(a)}{\Gamma(\nu+1)} (D+\sigma)\left[(x-a)^\nu e^{-\sigma (x-a)}\right]\\
  &\quad=I_s^\nu [D_s f(x)]+\frac{f(a)}{\Gamma(\nu)}(x-a)^{\nu-1} e^{-\sigma (x-a)}.
 \end{split}
 \end{equation*}
Hence, we get (\ref{2.2}).
\end{proof}

\begin{lemma} \label {lemma2.6}
Let $f(x)$ be (m-1)-times continuously differentiable on $(a,\infty)$ and its m-times derivative be integrable on any finite subinterval of  $[a,\infty)$, $\mu>0$, $\nu>0$; and  $m$ is the smallest integer that exceeds $\mu$.  Then  for all $x\geq a$,
\begin{equation}\label{2.3}
  I_s^{\nu}f(x)=I_s^{m+\nu}[D_s^m f(x)]+\sum _{k=0}^{m-1}\frac{D_s^{k}f({a})(x-a)^{k+\nu}e^{-\sigma (x-a)}}{\Gamma{(k+\nu+1)}};
\end{equation}
and
\begin{equation}\label{2.4}
\begin{split}
   D_s^\mu f(x)&=I_s^{m-\mu} [D_s^m f(x)]+\sum _{k=0}^{m-1}\frac{D_s^{k}f({a})(x-a)^{k-\mu}e^{-\sigma (x-a)}}{\Gamma{(k-\mu+1)}}\\
   &={^C\!D}_s^\mu f(x) +\sum _{k=0}^{m-1}\frac{D_s^{k}f({a})(x-a)^{k-\mu}e^{-\sigma (x-a)}}{\Gamma{(k-\mu+1)}},
\end{split}
\end{equation}
where ${^C\!D}_s^\mu f(x)=I_s^{m-\mu} [D_s^m f(x)]$ can be similarly called Caputo fractional substantial derivative \cite{Podlubny:99}.
In particular, from (\ref{2.3}) and (\ref{2.4}), we can extend the definitions of $I_s^{\nu}$ and $D_s^{\mu}$, i.e., $\mu, \nu$ can belong to $\mathbb{R}$ instead of being limited to $\mathbb{R}^+$, then for any real $\alpha$, there exists
\begin{equation}\label{2.5}
 I_s^\alpha=D_s^{-\alpha}.
\end{equation}
\end{lemma}

\begin{proof}
Replacing $\nu$ by $\nu+1$ and $f$ by $D_s f$ in (\ref{2.1}), we obtain
$$I_s^{\nu+1}[D_sf(x)]=I_s^{\nu+2}[D_s^2 f(x)]+\frac{D_sf(a)}{\Gamma(\nu+2)}(x-a)^{\nu+1} e^{-\sigma (x-a)}.$$
Thus, according to the above equation and (\ref{2.1}), there exists
\begin{equation*}
  \begin{split}
    I_s^\nu f(x)& =I_s^{\nu+1}[D_sf(x)]+\frac{f(a)}{\Gamma(\nu+1)}(x-a)^\nu e^{-\sigma (x-a)} \\
                & =I_s^{\nu+2}[D_s^2 f(x)]+\frac{D_sf(a)}{\Gamma(\nu+2)}(x-a)^{\nu+1} e^{-\sigma (x-a)}+\frac{f(a)}{\Gamma(\nu+1)}(x-a)^\nu e^{-\sigma (x-a)} \\
                &=I_s^{(\nu+m)}[D_s^m f(x)]+\sum _{k=0}^{m-1}\frac{D_s^{k}f({a})(x-a)^{\nu+k}e^{-\sigma (x-a)}}{\Gamma{(\nu+k+1)}}.
  \end{split}
\end{equation*}

To prove (\ref{2.4}), letting $D_s$ perform on both sides of (\ref{2.2}) leads to
 $$D^2_s[I_s^\nu f(x)]=D_s \{I_s^\nu [D_s f(x)]\}+\frac{f(a)}{\Gamma(\nu-1)}(x-a)^{\nu-2} e^{-\sigma (x-a)},$$
and replacing $f$ with $D_sf$ in (\ref{2.2}) yields
$$ D_s \{I_s^\nu [D_s f(x)]\}=I_s^\nu [D_s^2 f(x)]+\frac{D_sf(a)}{\Gamma(\nu)}(x-a)^{\nu-1} e^{-\sigma (x-a)}.$$
Therefore, there exists
\begin{equation*}
  \begin{split}
D^2_s[I_s^\nu f(x)]\!=I_s^\nu [D_s^2 f(x)]\!+\!\frac{D_sf(a)}{\Gamma(\nu)}(x-a)^{\nu-1} e^{-\sigma (x-a)}\! +\!\frac{f(a)}{\Gamma(\nu-1)}(x-a)^{\nu-2} e^{-\sigma (x-a)}.
  \end{split}
    \end{equation*}
Repeating the procedure $m-1$ times results in
\begin{equation}\label{2.6}
  D_s^m[I_s^\nu f(x)]=I_s^\nu [D_s^m f(x)]+\sum _{k=0}^{m-1}\frac{D_s^{k}f({a})(x-a)^{\nu+k-m}e^{-\sigma (x-a)}}{\Gamma{(\nu+k-m+1)}}.
\end{equation}
Taking $\nu=m-\mu$, then Eq. (\ref{2.6}) can be rewritten as
 $$D_s^\mu f(x)=D_s^m[I_s^\nu f(x)]=I_s^{m-\mu} [D_s^m f(x)]+\sum _{k=0}^{m-1}\frac{D_s^{k}f({a})(x-a)^{k-\mu}e^{-\sigma (x-a)}}{\Gamma{(k-\mu+1)}}.$$
From (\ref{2.3}) and (\ref{2.4}), it yields that $I_s^\alpha=D_s^{-\alpha}$ for any real $\alpha$.
\end{proof}

\begin{lemma}\label {lemma2.7}
Let $f(x)$ be (m-1)-times continuously differentiable on $(a,\infty)$ and its m-times derivative be integrable on any finite subinterval of  $[a,\infty)$ and $\nu>0$, where $m$ is the smallest integer that exceeds $\nu$. Then for all  $x>a$,
$$I_s^\nu[D_s^\nu f(x)]=f(x)-\sum_{j=1}^m[D_s^{\nu-j}f(x)]_{x=a}\frac{ \left(x-a\right)^{\nu-j} e^{-\sigma(x-a)}  }{\Gamma(\nu-j+1)}.$$
\end{lemma}
\begin{proof}
On the one hand, there exists
\begin{equation}\label{2.7}
\begin{split}
I_s^\nu[D_s^\nu f(x)]=D_s\left\{\frac{1}{\Gamma(\nu+1)}\int_{a}^x{\left(x-\tau\right)^{\nu}}e^{-\sigma(x-\tau)}[D_s^\nu f(\tau)] d\tau \right \}.
\end{split}
\end{equation}
On the other hand, repeatedly integrating by parts and  using Lemma \ref{lemma2.2} we have
\begin{equation}\label{2.8}
\begin{split}
& \frac{1}{\Gamma(\nu+1)}\int_{a}^x{\left(x-\tau\right)^{\nu}}e^{-\sigma(x-\tau)} D_s^\nu f(\tau) d\tau \\
&\quad =\frac{1}{\Gamma(\nu+1)}\int_{a}^x{\left(x-\tau\right)^{\nu}}e^{-\sigma(x-\tau)}D_s^m[I_s^{m-\nu} f(\tau)] d\tau \\
&\quad =\frac{1}{\Gamma(\nu)}\int_{a}^x{\left(x-\tau\right)^{\nu-1}}e^{-\sigma(x-\tau)}D_s^{m-1}[I_s^{m-\nu} f(\tau)] d\tau \\
&\qquad -\frac{ \left(x-a\right)^{\nu} e^{-\sigma(x-a)}  }{\Gamma(\nu+1)} \left \{D_s^{m-1}[I_s^{m-\nu}f(x)]\right \}_{x=a}\\
&\quad =\frac{1}{\Gamma(\nu-m+1)}\int_{a}^x{\left(x-\tau\right)^{\nu-m}}e^{-\sigma(x-\tau)}[I_s^{m-\nu} f(\tau)] d\tau \\
&\qquad -\sum_{j=1}^m \left \{D_s^{m-j}[I_s^{m-\nu}f(x)]\right \}_{x=a}   \frac{ \left(x-a\right)^{\nu-j+1} e^{-\sigma(x-a)}  }{\Gamma(\nu-j+2)}\\
&\quad =I_s^{\nu-m+1}[I_s^{m-\nu} f(\tau)]
         -\sum_{j=1}^m \left [D_s^{\nu-j}f(x)\right]_{x=a}   \frac{ \left(x-a\right)^{\nu-j+1} e^{-\sigma(x-a)}  }{\Gamma(\nu-j+2)}\\
&\quad =I_s f(\tau)-\sum_{j=1}^m \left [D_s^{\nu-j}f(x)\right]_{x=a}   \frac{ \left(x-a\right)^{\nu-j+1} e^{-\sigma(x-a)}  }{\Gamma(\nu-j+2)}.
\end{split}
\end{equation}
Combining \ref{2.7} and \ref{2.8}, we obtain
$$I_s^\nu[D_s^\nu f(x)]=f(x)-\sum_{j=1}^m[D_s^{\nu-j}f(x)]_{x=a}\frac{ \left(x-a\right)^{\nu-j} e^{-\sigma(x-a)}  }{\Gamma(\nu-j+1)}.$$
\end{proof}

\begin{lemma} \label {lemma2.8}
Let $f(x)$ be (m-1)-times continuously differentiable on $(a,\infty)$ and its m-times derivative be integrable on any finite subinterval of  $[a,\infty)$ and  $\mu>0,\nu >0$, where $m$ is the smallest integer that exceeds $\mu$. Then for all  $x>a$,
$$D_s^\mu[D_s^{-\nu} f(x)]=D_s^{\mu-\nu}f(x).$$
\end{lemma}
\begin{proof}
Two cases must be considered: $\mu > \nu \geq 0$ and $\nu \geq \mu \geq 0$.

Case $\mu > \nu \geq 0$:   taking $0\leq n-1 \leq \mu-\nu <n$, $n$ is an integer and  using  $0 \leq m-1 \leq \mu <m$, then from (\ref{2.5})  and (\ref{1.2}) and Lemmas  \ref{lemma2.2} and \ref{lemma2.4},
 we have
 \begin{equation*}
\begin{split}
D_s^\mu[D_s^{-\nu} f(x)]&=D_s^\mu[I_s^\nu f(x)]=D_s^{m}\left\{I_s^{m-\mu}[I_s^{\nu}f(x)] \right\}
=D_s^{m}\left\{I_s^{m-\mu+\nu}f(x) \right\} \\
&=D_s^{n}\left\{I_s^{n-\mu+\nu}f(x) \right\}=D_s^{\mu-\nu}f(x).
\end{split}
\end{equation*}

Case $\nu \geq \mu \geq 0$: according to Lemmas \ref{lemma2.2} and \ref{lemma2.3}, we obtain
$$D_s^\mu[I_s^\nu f(x)]=D_s^\mu[I_s^\mu I_s^{\nu-\mu}f(x)]= D_s^{\mu-\nu}f(x).$$
\end{proof}

\begin{lemma} \label {lemma2.9}
Let $f(x)$ be (m-1)-times continuously differentiable on $(a,\infty)$ and its m-times derivative be integrable on any finite subinterval of  $[a,\infty)$ and  $\mu>0,\nu >0$, where $m$ is the smallest integer that exceeds $\nu$. Then for all  $x>a$,
$$D_s^{-\mu}[D_s^\nu f(x)]=D_s^{\nu-\mu}f(x)-\sum_{j=1}^m[D_s^{\nu-j}f(x)]_{x=a}\frac{ \left(x-a\right)^{\mu-j} e^{-\sigma(x-a)}  }{\Gamma(\mu-j+1)}.$$
\end{lemma}
\begin{proof}
If $\nu \leq \mu$, there exists  $D_s^{-\mu}=D_s^{\nu-\mu}D_s^{-\nu}$ by Lemma \ref{lemma2.2};  and
if $\nu \geq \mu$, there  also exists  $D_s^{-\mu}=D_s^{\nu-\mu}D_s^{-\nu}$ by Lemma \ref{lemma2.8}. Therefore, using Lemma \ref{lemma2.7} we have
\begin{equation*}
\begin{split}
D_s^{-\mu}[D_s^\nu f(x)]
&=D_s^{\nu-\mu}\{ D_s^{-\nu}[D_s^\nu f(x)]\} \\
&=D_s^{\nu-\mu}\left\{ f(x)-\sum_{j=1}^m[D_s^{\nu-j}f(x)]_{x=a}\frac{ \left(x-a\right)^{\nu-j} e^{-\sigma(x-a)}  }{\Gamma(\nu-j+1)} \right\} \\
&=D_s^{\nu-\mu}f(x)-\sum_{j=1}^m[D_s^{\nu-j}f(x)]_{x=a}\frac{ \left(x-a\right)^{\mu-j} e^{-\sigma(x-a)}  }{\Gamma(\mu-j+1)},
\end{split}
\end{equation*}
where we use the following formula
\begin{equation}\label{2.9}
D_s^\mu [ e^{-\sigma(x-a)} \left(x-a\right)^{\nu}]=\frac{\Gamma (\nu+1)}{\Gamma (\nu+1-\mu)}\left(x-a\right)^{\nu-\mu}e^{-\sigma(x-a)},
\end{equation}
which can be similarly proven as the way in \cite[p. 56]{Podlubny:99}.
\end{proof}

\begin{lemma}\label {lemma2.10}
Let $\mu>0$, $\nu>0$ and $f(x)$ be (r-1)-times continuously differentiable on $(a,\infty)$ and its r-times derivative be integrable on any finite subinterval of  $[a,\infty)$, where $r=\max (m,n)$,
$m$ and $n$ is the smallest integer that exceeds $\mu$ and  $\nu$, respectively. Then for all  $x>a$,
$$D_s^{\mu}[D_s^\nu f(x)]=D_s^{\mu+\nu}f(x)-\sum_{j=1}^n[D_s^{\nu-j}f(x)]_{x=a}\frac{ \left(x-a\right)^{-\mu-j} e^{-\sigma(x-a)}  }{\Gamma(-\mu-j+1)}.$$
\end{lemma}
\begin{proof}
Similar to the well-known property of integer-order derivatives:
$$\frac{d^m}{dx^m} \left(\frac{d^n f(x)}{dx^n} \right)=\frac{d^n}{dx^n} \left(\frac{d^m f(x)}{dx^m} \right)=\frac{d^{m+n}f(x)}{dx^{m+n}}, $$
it is easy to check that
$$D_s^m \left[D_s^nf(x)\right]=D_s^n \left[D_s^mf(x)\right]=D_s^{m+n}f(x).$$
Therefore, according to (\ref{1.2}), the above equation, and Lemma \ref {lemma2.8}, there exists
\begin{equation*}
 D_s^n \left[D_s^{m-\alpha} f(x)\right]=D_s^{n+m}[I_s^\alpha f(x)]=D_s^{n+m-\alpha}f(x), ~~{\rm for}~~\alpha \in (0,1],
\end{equation*}
and denoting that $\gamma=m-\alpha$, it leads to
\begin{equation}\label{2.10}
 D_s^n \left[D_s^{\gamma} f(x)\right]=D_s^{n+\gamma}f(x).
\end{equation}

According to (\ref{1.2}), Lemma \ref {lemma2.9}, and (\ref{2.10}), we obtain
\begin{equation*}
\begin{split}
 D_s^{\mu}[D_s^\nu f(x)]&=D_s^m \left \{D_s^{-(m-\mu)} [D_s^\nu f(x)]           \right \}  \\
  &=D_s^m \left \{  D_s^{\mu+\nu-m}f(x)
  -\sum_{j=1}^n[D_s^{\nu-j}f(x)]_{x=a}\frac{ \left(x-a\right)^{m-\mu-j} e^{-\sigma(x-a)}  }{\Gamma(m-\mu-j+1)}  \right \} \\
  &=D_s^{\mu+\nu}f(x)-\sum_{j=1}^n[D_s^{\nu-j}f(x)]_{x=a}\frac{ \left(x-a\right)^{-\mu-j} e^{-\sigma(x-a)}  }{\Gamma(-\mu-j+1)}.
\end{split}
\end{equation*}
\end{proof}

Similar to the proof of \cite[p. 76-77]{Podlubny:99}, we have the following Remarks.
\begin{remark}
If  $D_s^\mu f(x)$ exists and is integrable,
then the fractional substantial derivative $D_s^\nu f(x)$ also exists and is integrable for  $0<\nu<\mu$.
\end{remark}

\begin{remark}
Let $f(x)$  be (m-1)-times continuously differentiable on $(a,\infty)$ and its m-times derivative be integrable on any finite subinterval of  $[a,\infty)$, then for all $x\geq a$,
$$[D_s^\mu f(x)]_{x=a}=0,~~m-1 \leq \mu <m$$ if and only if
$$D_s^{(j)}=0, ~~{\rm for}~~j=0,1,\ldots, m-1.$$
\end{remark}

\section{Discretizations of fractional substantial calculus and its convergence; Fourier transform methods}
In this section, we derive the discretization schemes of fractional substantial calculus and prove their convergence by Fourier transform method.
\begin{lemma}  \label{lemma3.1}
 Let  $\nu>0$, $f(x) \in L^q(\mathbb{R})$, $q \geq 1$, and
\begin{equation}\label{3.1}
\mathcal{I}_s^\nu f(x)=\frac{1}{\Gamma(\nu)}\int_{{-\infty}}^x{\left(x-\tau\right)^{\nu-1}}e^{-\sigma(x-\tau)}{f(\tau)}d\tau,
\end{equation}
then
\begin{eqnarray*}
 \mathcal{F}(\mathcal{I}_s^\nu f(x))=(\sigma-i\omega)^{-\nu}\widehat{f}(\omega),
\end{eqnarray*}
where $\mathcal{F}$ denotes Fourier transform operator and $\widehat{f}(\omega)=\mathcal{F}(f)$, i.e.,
\begin{eqnarray*}
    \widehat{f}(\omega)=\int_{\mathbb{R}}e^{i\omega x }f(x)dx.
 \end{eqnarray*}
\end{lemma}

\begin{proof}
Taking the fractional substantial integral (\ref{1.1}) with the lower terminal $a=-\infty$, Eq. (\ref{1.1})
reduces to (\ref{3.1}).

Let us start with the Laplace transform of the function
$$h(x)=\frac{x^{\nu-1}}{\Gamma(\nu) }e^{-\sigma x},$$
i.e.,
\begin{equation}\label{3.2}
\begin{split}
\frac{1}{\Gamma(\nu)}
 \int_{0}\nolimits^\infty{x^{\nu-1}}{e^{-(\sigma +s)x}}dx=(\sigma +s)^{-\nu},
\end{split}
\end{equation}
where we use the well-known Laplace transform of the function $x^{\nu-1}$
$$L\{x^{\nu-1};s\}=\int_{0}\nolimits^\infty x^{\nu-1}e^{-sx}dx=\Gamma(\nu)s^{-\nu}.$$
 It follows from the Dirichlet theorem \cite[p. 564]{Fikhtengoltz:69} that the integral (\ref{3.2}) converges if $\nu>0$. Taking $s=-i\omega$, where $\omega$ is real, we immediately have the Fourier transform of the function
 \begin{equation*}
\begin{split}
h_{+}(x)= \left\{ \begin{array}
 {l@{\quad} l}
\frac{x^{\nu-1}}{\Gamma(\nu) }e^{-\sigma x},&x> 0;\\
0,& x\leq0,
 \end{array}
 \right.
  \end{split}
\end{equation*}
in the form
\begin{equation*}
\begin{split}
\mathcal{F}(h_{+}(x))=
 \int_{-\infty}\nolimits^\infty h_{+}(x)e^{i\omega x }dx
=\frac{1}{\Gamma(\nu)}
 \int_{0}\nolimits^\infty{x^{\nu-1}}{e^{-(\sigma -i\omega)x}}dx=(\sigma -i\omega)^{-\nu}.
\end{split}
\end{equation*}
 Since
   \begin{equation*}
 \mathcal{I}_s^\nu f(x)=\frac{1}{\Gamma(\nu)}\int_{-\infty}^x{\left(x-\tau\right)^{\nu-1}}e^{-\sigma (x-\tau)}{f(\tau)}d\tau
 =\frac{x^{\nu-1}e^{-\sigma x}}{\Gamma(\nu) }*f(x)=h(x)*f(x),
\end{equation*}
where the asterisk means the convolution, then we have
  \begin{equation*}
   \begin{split}
 &\mathcal{F}(\mathcal{I}_s^\nu f(x))=\mathcal{F}(h(x)*f(x))=\mathcal{F}(h(x))\cdot \mathcal{F}(f(x))=(\sigma -i\omega)^{-\nu}\widehat{f}(\omega).
   \end{split}
 \end{equation*}

\end{proof}

\begin{lemma}  \label{lemma3.2}
 Let  $\nu>0$, $f \in C_0^{m-1}(a,\infty)$ and its m-times derivative be integrable on any finite subinterval of  $[a,\infty)$.  Denoting that
\begin{equation}\label{3.3}
  \mathcal{D}_s^\nu f(x)={D}_s^m[\mathcal{I}_s^{m-\nu} f(x)],
\end{equation}
where $m$ is the smallest integer that exceeds $\nu$
and  ${D}_s^m$ and  $\mathcal{I}_s^{m-\nu}$ are  defined by (\ref{1.3}) and  (\ref{3.1}), respectively. Then
\begin{eqnarray*}
 \mathcal{F}(\mathcal{D}_s^\nu f(x))=(\sigma-i\omega)^{\nu}\widehat{f}(\omega).
\end{eqnarray*}
\end{lemma}

\begin{proof}
  Taking the lower terminal $a=-\infty$ and  using  (\ref{2.4}), we obtain
  \begin{equation*}
\mathcal{D}_s^\nu f(x)= {D}_s^m[\mathcal{I}_s^{m-\nu} f(x)] = \mathcal{I}_s^{m-\nu}[{D}_s^m f(x)].
\end{equation*}
Then from Lemma \ref{lemma3.1}, there exists
\begin{eqnarray*}
 \mathcal{F}(\mathcal{D}_s^\nu f(x))
 =(\sigma-i\omega)^{\nu-m}\mathcal{F}( {D}_s^m f(x))=(\sigma-i\omega)^{\nu}\widehat{f}(\omega),
\end{eqnarray*}
where $\mathcal{F}( {D}_s^m f(x))=(\sigma-i\omega)^{m}\widehat{f}(\omega)$ can be proven by the mathematical induction.
\end{proof}

%
%First, taking  $p=1$, $\sigma=0$, for all $|\zeta| \leq 1$, Eq. (\ref{1.2}) becomes the following equation  \cite{Podlubny:99},
%\begin{equation}\label{2.3}
%\begin{split}
%(1-\zeta)^{\alpha}= \sum_{m=0}^{\infty}(-1)^m\left ( \begin{matrix}\alpha \\ m\end{matrix} \right )\zeta^m= \sum_{m=0}^{\infty}{g}_m\zeta^m,
%\end{split}
%\end{equation}
%with the recursively formula
%\begin{equation}\label{2.4}
%  {g}_0=1, ~~~~{g}_m=\left(1-\frac{\alpha+1}{m}\right)g_{m-1},~~m \geq 1.
%\end{equation}

In the following,  we do the expansions to (\ref{1.4}) to get the formulas of the coefficients when $p=1,2,3,4,5$; and we  prove that the operators have their respective desired convergent order by the technique of Fourier transform.

First, taking $p=1$ and $h$ be the uniform space stepsize, then from (\ref{1.4}), we have
\begin{equation*}
\begin{split}
\kappa^{1,\alpha}(\zeta)=(1-\frac{\zeta}{e^{\sigma h}})^{\alpha}
= \sum_{m=0}^{\infty}e^{-m\sigma h}(-1)^m\left ( \begin{matrix} \alpha \\ m\end{matrix} \right )\zeta^m
= \sum_{m=0}^{\infty}{g}_m^{1,\alpha}\zeta^m,
\end{split}
\end{equation*}
with the recursively formula
\begin{equation}\label{3.4}
  {g}_0^{1,\alpha}=1, ~~~~{g}_m^{1,\alpha}=e^{-\sigma h}\left(1-\frac{\alpha+1}{m}\right){g}_{m-1}^{1,\alpha},~~m \geq 1,
\end{equation}
where $\sigma $ is defined in Definition \ref{definition1.1}.

Similar to the way performed in \cite{Chen:13}, it is easy to compute
\begin{equation}\label{3.5}
\kappa^{p,\alpha}(\zeta) = \left(\sum_{i=1}^p\frac{1}{i}\left(1-\frac{\zeta}{e^{\sigma h}}\right)^i\right)^{\alpha}
= \sum_{m=0}^{\infty}{g}_m^{p,\alpha}\zeta^m,~~~~p=1,2,3,4,5,
\end{equation}
with ${g}_m^{1,\alpha}$ given in (\ref{3.4});
%and
%\begin{equation*}\
%\begin{split}
%& {g}_m^{2,\alpha} =\left(\frac{3}{2}\right)^{\alpha} \sum_{i=0}^{m} \mu_2^{i}\,{g}_i^{1,\alpha}\,{g}_{m-i}^{1,\alpha},\\
%& {g}_m^{3,\alpha}=\left(\frac{11}{6}\right)^{\alpha} \sum_{j=0}^{m}\sum_{i=0}^{j}{\mu_3}\!^i \,\overline{{\mu_3}}^{j-i}\,{g}_i^{1,\alpha}\,{g}_{j-i}^{1,\alpha}\,{g}_{m-j}^{1,\alpha}, \\
%&{g}_m^{4,\alpha}
%  =\left(\frac{25}{12}\right)^{\alpha}\sum_{k=0}^{m} \sum_{j=0}^{k}\sum_{i=0}^{j}{\nu_4}\!^{m-k} {\mu_4}\!^i \,\overline{{\mu_4}}^{j-i}\,{g}_i^{1,\alpha}\,{g}_{j-i}^{1,\alpha}\,{g}_{k-j}^{1,\alpha}\,{g}_{m-k}^{1,\alpha},\\
%&{g}_m^{5,\alpha}
%  =\left(\frac{137}{60}\right)^{\alpha}\sum_{n=0}^{m}\sum_{k=0}^{n} \sum_{j=0}^{k}\sum_{i=0}^{j}\overline{\nu_5}\!^{m-n}{\nu_5}\!^{n-k} {\mu_5}\!^i \,\overline{{\mu_5}}^{j-i}{\,g}_i^{1,\alpha}{\,g}_{j-i}^{1,\alpha}{\,g}_{k-j}^{1,\alpha}{\,g}_{n-k}^{1,\alpha}{\,g}_{m-n}^{1,\alpha},\\
%\end{split}
%\end{equation*}
%where $\mu_2$; $\mu_3$; $\nu_4$, $\mu_4$; $\nu_5$, $\mu_5$ are defined in  (2.4), (2.6), (2.8) and (2.10) of  \cite{Chen:13}, respectively.
and
$${g}_m^{p,\alpha}=e^{-\sigma mh}{l}_m^{p,\alpha},~~~~p=1,2,3,4,5,$$
where ${l}_m^{1,\alpha}$, ${l}_m^{2,\alpha}$,    ${l}_m^{3,\alpha}$, ${l}_m^{4,\alpha}$ and  ${l}_m^{5,\alpha}$
are defined by (2.2),  (2.4), (2.6), (2.8) and (2.10) in \cite{Chen:13}, respectively. And it implies that
to get the coefficients ${g}_m^{p,\alpha}$, we only need  to compute the coefficients ${l}_m^{p,\alpha}$.

\begin{theorem}(Case $p=1$)\label{theorem3.3}
   Let  $f$, $\mathcal{D}_s^{\alpha+1} f(x)$ with $\alpha>0$  and their Fourier transforms belong to $L_1(\mathbb{R})$, and  denote that
 \begin{equation*}
A^{1,\alpha}f(x)=\frac{1}{h^{\alpha}}\sum_{m=0}^{\infty}{g}_m^{1,\alpha}f(x-mh),
\end{equation*}
 where $\mathcal{D}_s^{\alpha+1}$ and ${g}_m^{1,\alpha}$ is defined by (\ref{3.3}) and (\ref{3.4}), respectively,  Then
$$
 \mathcal{D}_s^\alpha f(x)=A^{1,\alpha}f(x)+\mathcal{O}(h).
$$
\end{theorem}

\begin{proof}
Using Fourier transform, we obtain
\begin{equation*}
\begin{split}
\mathcal{F}(A^{1,\alpha}f)(\omega)&=\frac{1}{h^{\alpha}}\sum_{m=0}^{\infty}{g}_m^{1,\alpha}\mathcal{F}\left(f(x-mh)\right)(\omega) \\
&=\frac{1}{h^{\alpha}} \sum_{m=0}^{\infty}{g}_m^{1,\alpha} \left(e^{i\omega h}\right)^m \widehat{f}(\omega)\\
&=\frac{1}{h^{\alpha}} \left(1-\frac{e^{i\omega h}}{e^{\sigma h}}\right)^{\alpha}\widehat{f}(\omega)\\ %%%%%%%%
%&=\frac{1}{h^{\alpha}}\left(1-e^{-(\sigma -i\omega) h}\right)^{\alpha}\widehat{f}(\omega)\\ %%%%%%%%
&=(\sigma -i\omega)^{\alpha} \left(\frac{1-e^{-(\sigma -i\omega) h}}{ (\sigma -i\omega) h }\right )^{\alpha}\widehat{f}(\omega)\\
&=(\sigma -i\omega)^{\alpha} \left(\frac{1-e^{-z}}{z}\right )^{\alpha}\widehat{f}(\omega),
  \end{split}
\end{equation*}
with $z=(\sigma -i\omega) h$. It is easy to check that
\begin{equation*}
\begin{split}
\left(\frac{1-e^{-z}}{z}\right )^{\alpha} =1- \frac{\alpha}{2}z +\frac{3\alpha^2+\alpha}{24} z^2 -\frac{\alpha^3+\alpha^2}{48} z^3 +\mathcal{O}(z^4).
\end{split}
\end{equation*}
Therefore, from Lemma \ref{lemma3.2},  there exists
\begin{equation*}
\begin{split}
\mathcal{F}(A^{1,\alpha}f)(\omega)=\mathcal{F}( \mathcal{D}_s^\alpha f)+ \widehat{\phi}(\omega),
  \end{split}
\end{equation*}
where $ \widehat{\phi}(\omega)=(\sigma -i\omega)^{\alpha}\left(-\frac{\alpha}{2}z+\mathcal{O}(z^2)\right)\widehat{f}(\omega)$,
$z=(\sigma -i\omega) h$. Then
\begin{equation*}
\begin{split}
&|\widehat{\phi}(\omega)| \leq \widetilde{c}\cdot|(\sigma -i\omega)^{\alpha+1}\widehat{f}(\omega)|\cdot h.
\end{split}
\end{equation*}
With the condition  $\mathcal{F}[\mathcal{D}_s^{\alpha+1} f(x)] \in L_1(\mathbb{R})$, it leads to
\begin{equation*}
\begin{split}
| \mathcal{D}_s^\alpha f(x)-A^{1,\alpha}f(x)|=|\phi(x)| \leq \frac{1}{2\pi}\int_{\mathbb{R}}|\widehat{\phi}(\omega)|dx
\leq c || \mathcal{F}[\mathcal{D}_s^{\alpha+1} f](\omega)||_{L^1}\cdot h= \mathcal{O}(h).
  \end{split}
\end{equation*}
\end{proof}

\begin{theorem}\label{theorem3.7}
   Let  $f$, $\mathcal{D}_s^{\alpha+p} f(x)$ ($p=2,3,4,5$) with $\alpha>0$  and their Fourier transforms belong to $L_1(\mathbb{R})$, and  denote that
 \begin{equation*}
A^{p,\alpha}f(x)=\frac{1}{h^{\alpha}}\sum_{m=0}^{\infty}{g}_m^{p,\alpha}f(x-mh),
\end{equation*}
 where ${g}_m^{p,\alpha}$ is defined by (\ref{3.5}).  Then
$$
 \mathcal{D}_s^\alpha f(x)=A^{p,\alpha}f(x)+\mathcal{O}(h^p),~~ p=2,3,4,5.
$$
\end{theorem}
\begin{proof}
Using the ideas of the proof of Theorem \ref{theorem3.3} and Lemmas 2.3-2.7 of \cite{Chen:13}, we can similarly prove this theorem; the details are omitted here.
\end{proof}

\begin{remark}
Theorems \ref{theorem3.3}-\ref{theorem3.7} still hold for the fractional substantial integral operators  $\mathcal{I}_s^\alpha$; in fact, comparing Lemmas \ref{lemma3.1} with \ref{lemma3.2} gives us the intuition.
\end{remark}

All the above schemes are applicable to finite domain, say, $(a, b)$, after performing zero extensions to the functions considered. Let $f(x)$ be the zero extended function from the finite domain $(a, b)$, and satisfy the requirements of the above corresponding theorems. Taking $p=1,2,3,4,5$ and
 \begin{equation}\label{3.6}
  \begin{split}
&\widetilde{A}^{p,\alpha}f(x)=\frac{1}{h^{\alpha}}\sum_{m=0}^{[\frac{x-a}{h}]}{g}_m^{p,\alpha}f(x-mh), ~~\alpha>0,
\end{split}
\end{equation}
with $g_m^{p,\alpha}$ given in (\ref{3.5}). Then
\begin{equation}\label{3.7}
  \begin{split}
    D_{s}^{\alpha}f(x)=\widetilde{A}^{p,\alpha}f(x)+\mathcal{O}(h^p), ~~\alpha>0,
  \end{split}
\end{equation}
where $D_s^\alpha$ is defined by (\ref{1.2}).
 Thus the approximation operator of (\ref{3.6}) can be  described as
 \begin{equation*}
\widetilde{A}^{p,\alpha}f(x_i)=\frac{1}{h^{\alpha}}\sum_{m=0}^{i}{g}_m^{p,\alpha}f(x_{i-m}),~~\alpha>0,
\end{equation*}
and  the fractional substantial derivative has $p$-th order  approximations
\begin{equation}\label{3.8}
  \begin{split}
    D_{s}^{\alpha}f(x_i)
    =h^{-\alpha}\sum_{m=0}^{i}{g}_m^{p,\alpha}f(x_{i-m})+\mathcal{O}(h^p), ~~\alpha>0,
  \end{split}
\end{equation}
Similarly,  the fractional substantial integral has $p$-th order  approximations
\begin{equation}\label{3.9}
  \begin{split}
    I_{s}^{\alpha}f(x_i)
    =h^{\alpha}\sum_{m=0}^{i}{g}_m^{p,-\alpha}f(x_{i-m})+\mathcal{O}(h^p), ~~\alpha>0.
  \end{split}
\end{equation}

\section{Discretizations of fractional substantial calculus and its convergence; fractional linear multistep methods}
Essentially the results given this section are the generalizations of the ones for fractional calculus provided in \cite{Lubich:86} to fractional substantial calculus; some of them are not straightforward, so we restate and prove them. In particular, comparing with Section 3, by adding some terms at the neighborhood of the boundary of the fractional substantial calculus, we can relax the regularity requirements of the performed functions but still preserve the desired convergent order.

For the simplicity of presentation, we take the lower terminal $a=0$ (that is not essential, $a$ can be any given constant but not infinity).
Then the fractional substantial integral (\ref{1.1}) and fractional substantial derivative (\ref{1.2}), respectively,  reduce to
\begin{equation}\label{4.1}
I_s^\alpha f(x)=\frac{1}{\Gamma(\alpha)}\int_{0}^x{\left(x-\tau\right)^{\alpha-1}}e^{-\sigma (x-\tau)}{f(\tau)}d\tau,
\end{equation}
and
\begin{equation}\label{4.2}
D_s^\alpha f(x)=D_s^m[I_s^{m-\alpha} f(x)],
\end{equation}
where $m$ is the smallest integer that exceeds $\alpha$.

For  $\sigma=0$,  the fractional substantial integral (\ref{4.1}) and fractional substantial derivative (\ref{4.2}), respectively,  reduce to the Riemann-Liouville fractional integral
\begin{equation}\label{4.3}
I^\alpha f(x)=\frac{1}{\Gamma(\alpha)}\int_{0}^x{\left(x-\tau\right)^{\alpha-1}}{f(\tau)}d\tau,
\end{equation}
and
Riemann-Liouville fractional derivative  \cite{Gorenflo:97,Miller:93,Oldham:74}
\begin{equation}\label{4.4}
D^\alpha f(x)=\frac{d^m}{dx^m}\frac{1}{\Gamma(m-\alpha)}\int_{0}^x{\left(x-\tau\right)^{m-\alpha-1}}{f(\tau)}d\tau,~~m-1<\alpha<m.
\end{equation}

Using the homogeneity and the convolution structure of  $I^\alpha$ in (\ref{4.3}):
$$
(I^\alpha f)(x)=x^\alpha (I^\alpha f(tx))(1)  ~~{\rm and }~~
I^\alpha f=\frac{1}{\Gamma(\alpha)} t^{\alpha-1}*f,
$$
Lubich gets the following important property \cite{Lubich:86}
\begin{equation}\label{4.5}
\left(E_{h}^\alpha t^{\beta-1}\right)(x)=x^{\alpha+\beta-1}\left(E_{h/x}^\alpha t^{\beta-1}\right)(1),
\end{equation}
with
\begin{equation}\label{4.6}
E_{h}^\alpha=\Omega_{h}^\alpha-I^\alpha ~~{\rm and}~~ \Omega_{h}^\alpha f(x)= h^\alpha \sum _{j=0}^n\omega_{n-j}^\alpha f(jh), ~~(x=nh),
\end{equation}
where $\omega_n^\alpha$ denotes the convolution quadrature weights.
So Lubich obtains  the following   convolution quadratures
 to approximation the Riemann-Liouville fractional integral
\begin{equation}\label{4.7}
  I^\alpha_{h}f(x)=h^\alpha \sum _{j=0}^n\omega_{n-j}^\alpha f(jh)+h^\alpha\sum_{j=1}^r\omega_{n,j}^\alpha f(jh),~~(x=nh),~~\alpha>0,
\end{equation}
where $\omega_{n,j}^\alpha$ denotes the starting quadrature weights. The added term $h^\alpha\sum_{j=1}^r\omega_{n,j}^\alpha f(jh)$ is mainly for keeping the accuracy when relaxing the requirement of the regularity of $f(x)$.  

For  $D^\alpha_{h}f(x)$ or $I^{-\alpha}_{h}f(x)$  in (\ref{4.7}) with $\alpha>0$, taking
$D^{(j)}f(0)=0$, $j=0,1,\ldots,m-1$, $m-1<\alpha<m$, then it yields  the convolution structure of  $D^\alpha$ in (\ref{4.4}):
\begin{equation*}
\begin{split}
D^\alpha f(x)&=\frac{d^m}{dx^m} \left[\frac{1}{\Gamma(m-\alpha)}\int_{0}^x{\left(x-\tau\right)^{m-\alpha-1}}{f(\tau)}d\tau \right] \\
             &=\frac{1}{\Gamma(m-\alpha)}\int_{0}^x{\left(x-\tau\right)^{m-\alpha-1}}(d^m{f(\tau)}/d\tau^m) d\tau =\frac{1}{\Gamma(m-\alpha)} x^{m-\alpha-1}*\frac{d^mf(x)}{dx^m},
\end{split}
\end{equation*}
and the homogeneity of  $D^\alpha$:
$$(D^\alpha f)(x)=x^{m-\alpha} (D^\alpha f(tx))(1),~~m-1<\alpha<m.$$
Therefore, we also obtain following  property
\begin{equation*}
\left(E_{h}^{-\alpha} t^{\beta-1}\right)(x)=x^{-\alpha+\beta-1}\left(E_{h/x}^{-\alpha} t^{\beta-1}\right)(1),~~~~ \beta>m,
\end{equation*}
where
\begin{equation*}
E_{h}^{-\alpha}=\Omega_{h}^{-\alpha}-D^\alpha ~~{\rm and}~~ \Omega_{h}^{-\alpha} f(x)= h^{-\alpha} \sum _{j=0}^n\omega_{n-j}^{-\alpha}f(jh), ~~(x=nh).
\end{equation*}
So similar to the discussions in \cite{Lubich:86}, we can also get the following scheme
 to approximate the Riemann-Liouville fractional derivative
\begin{equation}\label{4.8}
  D^\alpha_{h}f(x)=h^{-\alpha} \sum _{j=0}^n\omega_{n-j}^{-\alpha} f(jh)+h^{-\alpha}\sum_{j=1}^r\omega_{n,j}^{-\alpha} f(jh),~~(x=nh),~~\alpha>0.
\end{equation}
Form (\ref{4.7}) and (\ref{4.8}), there exists
\begin{equation}\label{4.9}
  I^\alpha_{h}f(x)=h^\alpha \sum _{j=0}^n\omega_{n-j} f(jh)+h^\alpha\sum_{j=1}^r\omega_{n,j} f(jh),~~(x=nh),~~\alpha \in \mathbb{R},
\end{equation}
where $\alpha>0$ corresponds to (\ref{4.7}) ($\omega_n=\omega_n^\alpha,\omega_{n,j}=\omega_{n,j}^\alpha$) and  $\alpha<0$ corresponds to (\ref{4.8}) ($\omega_n=\omega_n^{-\alpha},\omega_{n,j}=\omega_{n,j}^{-\alpha}$).

In this section, we mainly focus on the discretized fractional substantial calculus; 
for simplicity, the following notations are used:
\begin{equation}\label{4.10}
\begin{split}
E_{s,h}^\alpha=\Omega_{s,h}^\alpha-I_s^\alpha,~~{\rm where}~~\Omega_{s,h}^\alpha f(x)= h^\alpha \sum _{j=0}^n\kappa_{n-j}f(jh), ~~(x=nh),~~\alpha \in \mathbb{R},
\end{split}
\end{equation}
and it is easy to get the following properties for $\alpha \in \mathbb{R}$:
\begin{equation}\label{4.11}
\begin{split}
 &\left( I_s^\alpha [e^{-\sigma t}f(t)]\right)(x)=e^{-\sigma x}\left(I^\alpha [f(t)]\right)(x);\\
& \left(E_{s,h}^\alpha [e^{-\sigma t}f(t)]\right) (x)=e^{-\sigma x} \left( E_{h}^\alpha [f(t)]\right)(x),
 \end{split}
\end{equation}
where $\alpha>0$ corresponds to fractional substantial integral  and  $\alpha<0$ corresponds to fractional substantial derivative.

So we consider the following scheme to approximate the fractional substantial integral (\ref{4.1}) or the fractional substantial derivative (\ref{4.2})
\begin{equation}\label{4.12}
\begin{split}
  I_{s,h}^\alpha f(x)=&h^\alpha \sum _{j=0}^n\kappa_{n-j}f(jh)+h^\alpha\sum_{j=1}^r\kappa_{n,j}f(jh)),~~(x=nh),~~\alpha \in \mathbb{R},
\end{split}
\end{equation}
where
\begin{equation}\label{4.13}
  \kappa_j=e^{-j\sigma h}\omega_j, ~~\omega_j \mbox{ is defined by (\ref{4.9})},
\end{equation}
and $\kappa_j$  and $\kappa_{n,j}$ also denote the convolution quadrature weights  and the starting quadrature weights, respectively.

Given a sequence $\kappa=(\kappa_n)_0^\infty$ (or $\omega=(\omega_n)_0^\infty$) and take \cite{Lubich:86}
$$\kappa(\zeta)=\sum_{n=0}^{\infty}\kappa_n \zeta^n,~~~~ \Big( \mbox{or}~~~ \omega(\zeta)=\sum_{n=0}^{\infty}\omega_n \zeta^n\Big),$$
to be its generating power series.  
\begin{definition}\label{definition4.1}
A convolution quadrature $\kappa$ is stable (for $I_s^\alpha$) if
$$\kappa_n=\mathcal{O}(n^{\alpha-1}).$$
\end{definition}
\begin{definition}\label{definition4.2}
A convolution quadrature $\kappa$ is consistent of order $p$ (for $I_s^\alpha$) if
$$h^\alpha \kappa\left(e^{\sigma h}e^{-h}\right)=1+\mathcal{O}(h^{p}).$$
\end{definition}
\begin{definition}\label{definition4.3}
A convolution quadrature $\kappa$ is convergent of order $p$ (to $I_s^\alpha$) if
\begin{equation}\label{4.14}
  (E_{s,h}^\alpha [e^{-\sigma t}t^{\beta-1}])(1)=\mathcal{O}(h^{\beta})+\mathcal{O}(h^{p})~~{\rm for~ all}~~\beta \in  \mathbb{C},\beta\neq 0,-1,-2,\cdots.
\end{equation}
\end{definition}
\begin{lemma} \label {lemma4.4}
If  $ (E_{s,h}^\alpha [e^{-\sigma t}t^{k-1}])(1)=\mathcal{O}(h^{k})+\mathcal{O}(h^{p})$ for $k=1,2,3,\ldots,$
then $\kappa$ is consistent of order $p$. 
Moreover, $\kappa$ is consistent of order $p$ if and only if $\omega$ is consistent of order $p$.
\end{lemma}
\begin{proof}
According to (\ref{4.11}), we have
$$(E_{s,h}^\alpha [e^{-\sigma t}t^{k-1}])(1)=e^{-\sigma }(E_{h}^\alpha t^{k-1})(1),  $$
and it leads to
$$(E_{h}^\alpha t^{k-1})(1)=\mathcal{O}(h^{k})+\mathcal{O}(h^{p}),~~ {\rm for}~~ k=1,2,3,\ldots.$$
Then from Lemma 3.1 of \cite{Lubich:86}, we obtain
\begin{equation*}
 h^\alpha \omega (e^{-h})=1+\mathcal{O}(h^{p}), ~{\rm with }~~
  \omega(\zeta)=\sum_{n=0}^\infty \omega_n\zeta^n.
\end{equation*}
%%  Taking $f(t)=e^{(1-\sigma )(t-x)}$ on the interval $[0,x]$, then we have
%%  \begin{equation*}
%%    e_{s,h}(x)=\left(E_{s,h}^\alpha e^{(1-\sigma )(t-x)}\right)(x)=\left(E_{h}^\alpha e^{t-x}\right)(x)
%%    = h^\alpha \sum _{0\leq jh \leq x}\omega_je^{-jh}-(I^\alpha e^{t-x})(x).
%%  \end{equation*}
Using (\ref{4.13}), there exists
\begin{equation}\label{4.15}
\kappa(\zeta)=\sum_{n=0}^\infty \kappa_n\zeta^n =\sum_{n=0}^\infty  e^{-n\sigma h}\omega_n\zeta^n =  \omega  \left (\frac{\zeta}{e^{\sigma h}}\right).
 \end{equation}
Therefore
\begin{equation*}
 h^\alpha \kappa\left(e^{\sigma h}e^{-h}\right)= h^\alpha \omega (e^{-h})=1+\mathcal{O}(h^{p}),
\end{equation*}
and it means that $\kappa$ is consistent of order $p$ if and only if $\omega$ is consistent of order $p$.
\end{proof}

Using  (3.6) of \cite{Lubich:86} and (\ref{4.15}), we get
\begin{equation}\label{4.16}
\begin{split}
 \kappa(\zeta)= \omega  \left (\frac{\zeta}{e^{\sigma h}}\right)
 =&   \left(1-\frac{\zeta}{e^{\sigma h}}\right)^{-\alpha} \Big[c_0+c_1\left(1-\frac{\zeta}{e^{\sigma h}}\right)+c_2\left(1-\frac{\zeta}{e^{\sigma h}}\right)^2+ \cdots \\
  & +c_{N-1}\left(1-\frac{\zeta}{e^{\sigma h}}\right)^{N-1}+\left(1-\frac{\zeta}{e^{\sigma h}}\right)^{N}\widetilde{r}\left(\frac{\zeta}{e^{\sigma h}}\right)\Big],
 \end{split}
 \end{equation}
and
\begin{equation*}
\begin{split}
 \kappa(\zeta)= \omega  \left (\frac{\zeta}{e^{\sigma h}}\right)
 =&   \left(1-\frac{\zeta}{e^{\sigma h}}\right)^{-\alpha} \widetilde{\omega}\left(\frac{\zeta}{e^{\sigma h}}\right).
 \end{split}
 \end{equation*}
Therefore, we can characterize consistency in terms of the coefficients $c_i$.
\begin{lemma} \label {lemma4.5}
Let  $\sum \limits_{i=0}^\infty \gamma_i(1-\zeta)^i=\left(-\frac{\ln \zeta}{1-\zeta}\right)^{-\alpha}$.
Then
$\kappa$ is consistent of order $p$ 
if and only if the coefficients $c_i$ in (\ref{4.16}) satisfy
$$c_i=\gamma_i~~{\rm for }~~i=0,1,\ldots,p-1.$$
\end{lemma}
\begin{proof}
From  Lemma \ref{lemma4.4}, it implies that $\kappa$ is consistent of order $p$ if and only if $\omega$ is consistent of order $p$.
Thus, using Lemma 3.2 of \cite{Lubich:86}, the desired result is obtained. 
%Eq. (\ref{4.8}) implies that
%\begin{equation*}
%h^\alpha \kappa\left({e^{-(1-c)h}}\right)=h^\alpha \omega\left({e^{-h}}\right)=\left(  \frac{h}{1-e^{-h}}\right)^\alpha\widetilde{\omega}\left(e^{-h}\right)
%\end{equation*}
%is $1+\mathcal{O}(h^{p})$ if and only if
%\begin{equation*}
%  \widetilde{\omega}\left(e^{-h}\right)=\left(\frac{h}{1-e^{-h}}\right)^{-\alpha}+\mathcal{O}(h^{p})
%\end{equation*}
%if and only if
%\begin{equation*}
%  \widetilde{\omega}\left(\zeta\right)=\left(-\frac{\ln \zeta}{1-\zeta}\right)^{-\alpha}+\mathcal{O}((1-\zeta)^{p}).
%\end{equation*}
\end{proof}

Whether the method $\kappa$ is stable depends on the remainder in the expansion (\ref{4.16}), and (\ref{4.16})
can be rewritten as
\begin{equation}\label{4.17}
\begin{split}
 \kappa(\zeta)=
  & \left(1-\frac{\zeta}{e^{\sigma h}}\right)^{-\alpha} \Big[c_0+c_1\left(1-\frac{\zeta}{e^{\sigma h}}\right)+ \cdots +c_{N-1}\left(1-\frac{\zeta}{e^{\sigma h}}\right)^{N-1}\Big]\\
  & +\left(1-\frac{\zeta}{e^{\sigma h}}\right)^{N}r\left(\frac{\zeta}{e^{\sigma h}}\right),
 \end{split}
 \end{equation}
where $r(\zeta)= \left(1-\zeta\right)^{-\alpha}\widetilde{r}\left(\zeta\right)$.

\begin{lemma} \label {lemma4.6}
$\kappa$ is stable if and only if $\omega$ is stable; and $\omega$ is stable if and only if  the coefficients $r_n$ of $r(\zeta)$ in (\ref{4.17}) satisfy
\begin{equation*}
  r_n=\mathcal{O}(n^{\alpha-1}).
\end{equation*}
\end{lemma}
\begin{proof}
By Lemma 3.3 of \cite{Lubich:86}, we have $\omega$ is stable if and only if $r_n=\mathcal{O}(n^{\alpha-1})$.
 From (\ref{4.13}) and $e^{-j\sigma h}\in [e^{-{|\sigma|\, x}},e^{|\sigma|\, x}]$, $j=0,1,\ldots,n, x=nh$, it implies that
$\kappa$ is stable if and only if $\omega$ is stable.
\end{proof}

\begin{lemma} \label {lemma4.7}
Convergence implies stability. Moreover, $\kappa$ is convergent of order $p$ if and only if $\omega$ is convergent of order $p$.
\end{lemma}
\begin{proof}
According to (\ref{4.11}), we have
$$(E_{s,h}^\alpha [e^{-\sigma t}t^{\beta-1}])(1)=e^{-\sigma }(E_{h}^\alpha t^{\beta-1})(1), $$
and it implies that $\kappa$ is convergent of order $p$ if and only if $\omega$ is convergent of order $p$.
Hence, according to Lemma 3.4 of \cite{Lubich:86}, the desired result is got. 
\end{proof}

\begin{lemma} \label {lemma4.8}
Let $\alpha, \beta \in \mathbb{C}$, $\beta \neq 0,-1,-2,\cdots.$ If $\kappa$ is stable, then the convolution quadrature error of $e^{-\sigma t}t^{\beta-1}$
has the asymptotic expansion as 
\begin{equation*}
(E_{s,h}^\alpha [e^{-\sigma t}t^{\beta-1}])(1)=e^{-\sigma}\left(e_0+e_1h+\cdots+e_{N-1}h^{N-1}+\mathcal{O}(h^{N})+ \mathcal{O}(h^{\beta})  \right),
\end{equation*}
 and  the coefficients $e_j=e_j(\alpha,\beta,c_0,\cdots,c_j)$ depend analytically on $\alpha,\beta$ and the coefficients $c_0,\cdots,c_j$ of (\ref{4.17}).
\end{lemma}

\begin{proof}
 From Lemma \ref {lemma4.6}, $\kappa$ is stable if and only if $\omega$ is stable. According to (\ref{4.11}) and Lemma 3.5 of \cite{Lubich:86}, we get
\begin{equation*}
\begin{split}
(E_{s,h}^\alpha [e^{-\sigma t}t^{\beta-1}])(1)&=e^{-\sigma}(E_{h}^\alpha t^{\beta-1})(1)\\
&=e^{-\sigma}\left(e_0+e_1h+\cdots+e_{N-1}h^{N-1}+\mathcal{O}(h^{N})+ \mathcal{O}(h^{\beta})  \right).
\end{split}
\end{equation*}
\end{proof}

\begin{lemma} \label {lemma4.9}
Let $\Re (\alpha)>0$.
If $(E_{s,h}^\alpha [e^{-\sigma t}t^{p-1}])(1)=\mathcal{O}(h^{p})$,
then $(E_{s,h}^\alpha [e^{-\sigma t}t^{\beta-1}])(1)=\mathcal{O}(h^{p})$ for all $\Re (\beta)>p$.
\end{lemma}
\begin{proof}
 According to (\ref{4.11}), it leads to
 $$(E_{s,h}^\alpha [e^{\sigma t}t^{p-1}])(1)=e^{-\sigma}(E_{h}^\alpha t^{p-1})(1)=\mathcal{O}(h^{p}).$$
Then form Lemma 3.6 of \cite{Lubich:86}, we obtain $(E_{h}^\alpha t^{\beta-1})(1)=\mathcal{O}(h^{p})$ for all $\Re (\beta)>p$. Using (\ref{4.11}) again, there exists $(E_{s,h}^\alpha [e^{-\sigma t}t^{\beta-1}])(1)=\mathcal{O}(h^{p})$ for all $\Re (\beta)>p$.
\end{proof}

\begin{lemma} \label {lemma4.10}
Let $\Re (\alpha)>0$. There exist $\widetilde{\gamma}_0,\widetilde{\gamma}_1,\cdots$ (independent of $\kappa$) such that the following  holds for stable $\kappa$:
\begin{equation*}
(E_{s,h}^\alpha [e^{-\sigma t}t^{q-1}])(1)=\mathcal{O}(h^{q}), ~~{\rm for }~~ q=1,2,\cdots, p,
\end{equation*}
if and only if $c_i$ of (\ref{4.17}) satisfy
\begin{equation*}
  c_i=\widetilde{\gamma}_i, ~~{\rm for }~~ i=0,1,\cdots, p-1.
\end{equation*}
\end{lemma}
\begin{proof}
From Lemma \ref {lemma4.6}, $\kappa$ is stable if and only if $\omega$ is stable.
From (\ref{4.11}) and  Lemma 3.7 of \cite{Lubich:86}, there eixsts
 $$(E_{s,h}^\alpha [e^{-\sigma t}t^{q-1}])(1)=e^{-\sigma}(E_{h}^\alpha t^{q-1})(1)=\mathcal{O}(h^{q}),$$
 if and only if
  $$(E_{h}^\alpha t^{q-1})(1)=\mathcal{O}(h^{q}),~~{\rm for }~~ q=1,2,\cdots, p$$
  if and only if the coefficients $c_i$ of (\ref{4.17}) satisfy
\begin{equation*}
  c_i=\widetilde{\gamma}_i, ~~{\rm for }~~i=0,1,\cdots, p-1.
\end{equation*}

\end{proof}

\begin{lemma} \label {lemma4.11}
Let $\alpha \in \mathbb{R}$. $\kappa$ is convergent of order $p$,  if it is stable and consistent of order $p$.
\end{lemma}
\begin{proof}
According to Lemmas \ref{lemma4.6} and \ref{lemma4.4},  $\kappa$ is stable and consistent of order $p$ if and only if
 $\omega$ is stable and consistent of order $p$. Then from Lemma 3.8 of \cite{Lubich:86}, $\omega$  is  convergent of order $p$,
and  it leads to that $\kappa$ is also convergent of order $p$ by Lemmas \ref{lemma4.7}.
\end{proof}
\begin{theorem}\label {theorem4.12}
$\kappa$ is stable and consistent of order $p$ if and only if it is convergent of order $p$.
\end{theorem}
\begin{proof}
  From lemmas \ref{lemma4.4}, \ref{lemma4.7} and  \ref{lemma4.11}, we obtain it.
\end{proof}
\begin{theorem}\label {theorem4.13}
Let $\kappa$ satisfy (\ref{4.14}), and $f(x)=x^{\beta-1}g(x)$, where $\beta\neq 0,-1,-2,\cdots$, for $\alpha \ge 0$ and $\beta>\lceil-\alpha\rceil$ for $\alpha<0$; and $g(x)$ is sufficiently differentiable.
Then, there exists a starting quadrature $\kappa_{n,j}$, such that the approximation $I_{s,h}^\alpha f$ given by (\ref{4.12}) satisfies
 $$I_{s,h}^\alpha f(x)-I_s^\alpha f(x)=\mathcal{O}(h^p).$$
\end{theorem}
\begin{proof}
%Choosing the integer $m$ such that
%$$\Re(m+\beta-1) \leq p < \Re(m+\beta).$$
 A suitable starting quadrature can be chosen by putting
 $$I_{s,h}^\alpha [e^{-\sigma t}t^{q+\beta-1}](x)-I_s^\alpha [e^{-\sigma t}t^{q+\beta-1}](x)=0, ~~q=0,1,\cdots, m-1,$$
 where $m$ satisfies $\Re(m+\beta-1) \leq p < \Re(m+\beta)$; then the following holds
\begin{equation}\label{4.18}
h^\alpha \sum_{j=1}^m \kappa_{n,j}e^{-\sigma jh}(jh)^{q+\beta-1}+(E_{s,h}^\alpha [e^{-\sigma t}t^{q+\beta-1}])(1)=0,~~  nh=1.
\end{equation}
According to (\ref{4.18}), we have 
\begin{equation*}
\begin{split}
& \sum_{j=1}^m \kappa_{n,j}e^{-\sigma jh}j^{q+\beta-1}
=\frac{\Gamma (q+\beta)}{\Gamma (\alpha+q+\beta)}e^{-\sigma nh}n^{q+\alpha+\beta-1} -\sum_{j=1}^n \kappa_{n-j}e^{-\sigma jh}j^{q+\beta-1},
 \end{split}
\end{equation*}
this gives a Vandermonde type system for $\kappa_{n,j}$.
From (\ref{4.14}) and (\ref{4.18}), we have
$$\sum_{j=1}^m \kappa_{n,j}e^{-\sigma jh}j^{q+\beta-1} =\mathcal{O}(n^{\alpha-1});$$
then
\begin{equation*}
 \kappa_{n,j} =\mathcal{O}(n^{\alpha-1}).
\end{equation*}

Let $f(x)=x^{\beta-1}g(x)=e^{-\sigma x} x^{\beta-1}h(x)$, where  $h(x)=e^{\sigma x}g(x)$, and $g(x)$ is  sufficiently differentiable.
Let $\beta \in [d,d+1)$, $d$ is an integer, then $\gamma=\beta-d \in [0,1)$.

According to  Lemma \ref{lemma2.7}, there exists
\begin{equation*}
\begin{split}
  f(x)=&\sum_{q=0}^N\frac{D_s^{(q+\gamma-1)}f(0)  }{\Gamma(q+\gamma)} x^{q+\gamma-1} e^{-\sigma x}
       +\frac{1}{\Gamma(N+\gamma)}\left[\left(t^{N+\gamma-1}e^{-\sigma t}\right) \ast D_s^{(N+\gamma)}f\right](x).
\end{split}
\end{equation*}
If $\Re(N+\gamma-1)>p$ and additionally $\Re(N-p+\alpha+\gamma)>0$, then using (\ref{4.11}) and following the proof of
Theorem 2.4 in \cite{Lubich:86}, it is easy to get
 $$I_{s,h}^\alpha f(x)-I_s^\alpha f(x)=\mathcal{O}(e^{-\sigma x} x^{m-p+\alpha+\gamma-1}h^p) ~~\mbox{uniformly for bounded } x.$$

If $m$ in (\ref{4.18}) is replaced by $l\, (>m)$ with $\Re (l-p+\alpha+\gamma-1) \geq 0$, then the following for the corresponding starting quadrature weights holds 
$$ \kappa_{n,j} =\mathcal{O}(n^{l-1-p+\alpha+\gamma-1}), $$
and by the similar arguments performed above, we can prove that 
 $$I_{s,h}^\alpha f(x)-I_s^\alpha f(x)=\mathcal{O}(h^p) ~~\mbox{uniformly for bounded } x.$$
\end{proof}

\section{Numerical Results} 

We use two numerical examples to confirm that the theoretical results given in the above sections, including the fractional substantial derivatives and integrals.  The first example mainly verifies the numerical stability and convergent order; and the second one primarily focuses on illustrating that the starting quadrature numerically works very well for keeping the high order accuracy when the performed function becomes less regular.
 And the $ l_\infty$ norm is used to measure the numerical errors.
\begin{example}\end{example}
To numerically verify the truncation error given in Theorem \ref{theorem3.7}  in a bounded domain.
We utilize the approximation (\ref{3.8}) with $p=5$ to simulate the following equation
$$D_s^{\alpha}f(x)=\frac{\Gamma(6+\alpha)}{\Gamma(6)}x^{5}e^{-\sigma x},~~x \in(0,1),~~\sigma=1/2.$$
When $\alpha <0$, the fractional operator $D_s^{\alpha}$ becomes  fractional substantial integral operator; if $\alpha \in (0,1)$ we take $f(0)=0$; and if $\alpha \in (1,2)$ let $f(0)=0$, $f(1)=e^{-\sigma}$; the exact solution of the above equation is $f(x)=e^{-\sigma x}x^{5+\alpha}$.

\begin{table}[h]\fontsize{9.5pt}{12pt}\selectfont%生成浮动表格
 \begin{center}%\def\tabcolsep{28.5pt}%表格居中
  \caption {The maximum errors and convergent orders for  (\ref{3.8}), when $p=5$, $\sigma=1/2$.} \vspace{5pt}%标题，离表格一定的距离
\begin{tabular*}{\linewidth}{@{\extracolsep{\fill}}*{10}{c}}                                    \hline  %画顶端的横线
%$L=4,h$& $\alpha=-0.5$&  Rate       & $\alpha=0.5$  & Rate       & $\alpha=1.8$ &   Rate    \\\hline
%~~~1/10&  0.0000e-03  &             & 0.0000e-03    &            & 0.0000e-04   & 0.0000  \\
%~~~1/20&  0.0000e-04  &  0.0000     & 0.0000-04     & 0.0000     & 0.0000e-04   & 0.0000   \\
%~~~1/40&  0.0000e-04  &  0.0000     & 0.0000-04     & 0.0000     & 0.0000e-04   & 0.0000   \\
%~~~1/60&  0.0000e-04  &  0.0000     & 0.0000-04     & 0.0000     & 0.0000e-04   & 0.0000   \\ \hline
%%%%%%%%%%%%%%%%%%%%%%%%%%%%%%%%%%%%%%%%%%%%%%%%%%%%%%%%%%%%%%%%%%%%%%%%%%%%%%%%%%%%%%%%%%%%%%%%%%%%%%%%%%%%%%%%%%%%%%%%%%%%%%%%%%%%
$h$& $\alpha=-1/2$&  Rate       & $\alpha=1/2$  & Rate       & $\alpha=3/2$ &   Rate    \\\hline
~~~1/10&  3.7956e-005  &             & 2.0214e-004     &            & 3.7954e-003   &         \\
~~~1/20&  1.3109e-006  &  4.8557     & 6.9814e-006     & 4.8557     & 1.2933e-004   & 4.8751   \\
~~~1/40&  4.3065e-008  &  4.9279     & 2.2935e-007     & 4.9279     & 4.3193e-006   & 4.9041   \\
~~~1/80&  1.3798e-009  &  4.9639     & 7.3488e-009     & 4.9639     & 1.4014e-007   & 4.9459   \\
~~~\,~1/160& 4.3662e-011  &  4.9820     & 2.3254e-010     & 4.9820     & 4.4622e-009   & 4.9729   \\ \hline
    \end{tabular*}\label{example:1}%\vspace{-15pt}
  \end{center}
\end{table}

Table \ref{example:1} numerically verifies Theorem \ref{theorem3.7}, and shows that the truncation errors are $\mathcal{O}(h^5)$.

\begin{example}\end{example}
To numerically confirm the result given in Sec. 4 that the starting quadrature can keep the accuracy when the performed function is not sufficiently regular, we utilize the approximation (\ref{4.12}) and  (\ref{3.8}) (both with $p=5$), respectively, to simulate the following equation
$$D_s^{\alpha}f(x)=\frac{\Gamma(6+\alpha)}{\Gamma(6)}x^{5}e^{-\sigma x}+\frac{\Gamma(1.6)}{\Gamma(1.6-\alpha)}x^{0.6-\alpha}e^{-\sigma x},~~x \in(0,1),~~\sigma=1/2.$$
When $\alpha <0$, the fractional operator $D_s^{\alpha}$ is a fractional substantial integral operator; if $\alpha \in (0,1)$ we take $f(0)=0$; the exact solution of the above equation is $f(x)=e^{-\sigma x}(x^{5+\alpha}+x^{0.6}).$
\begin{table}[h]\fontsize{9.5pt}{12pt}\selectfont%生成浮动表格
 \begin{center}%\def\tabcolsep{28.5pt}%表格居中
  \caption {The maximum errors and convergent orders for (\ref{4.12}) and  (\ref{3.8}), respectively, when $p=5$, $\sigma=0.5$, $\beta=1.6$, $r=4$.} \vspace{5pt}% 标题，离表格一定的距离
\begin{tabular*}{\linewidth}{@{\extracolsep{\fill}}*{10}{c}}                                    \hline  %画顶端的横线
%$L=4,h$& $\alpha=-0.5$&  Rate       & $\alpha=0.5$  & Rate       & $\alpha=1.8$ &   Rate    \\\hline
%~~~1/10&  0.0000e-03  &             & 0.0000e-03    &            & 0.0000e-04   & 0.0000  \\
%~~~1/20&  0.0000e-04  &  0.0000     & 0.0000-04     & 0.0000     & 0.0000e-04   & 0.0000   \\
%~~~1/40&  0.0000e-04  &  0.0000     & 0.0000-04     & 0.0000     & 0.0000e-04   & 0.0000   \\
%~~~1/60&  0.0000e-04  &  0.0000     & 0.0000-04     & 0.0000     & 0.0000e-04   & 0.0000   \\ \hline
%%%%%%%%%%%%%%%%%%%%%%%%%%%%%%%%%%%%%%%%%%%%%%%%%%%%%%%%%%%%%%%%%%%%%%%%%%%%%%%%%%%%%%%%%%%%%%%%%%%%%%%%%%%%%%%%%%%%%%%%%%%%%%%%%%%
\multicolumn{9}{c}{ Numerical scheme (\ref{4.12})~~~~~~~~~~~~~~~~~~~~Numerical scheme  (\ref{3.8})}  \\ \hline
$h$  &$\alpha=-0.5$ &  Rate     & $\alpha=0.5$  & Rate       & $\alpha=-0.5$ &   Rate   & $\alpha=0.5$ &   Rate   \\\hline
1/10 &  2.8710e-05  &           & 3.7035e-04    &            & 1.4508e-02   &         & 4.3208e-01  &  \\
1/20 &  1.0424e-06  &  4.78     & 1.2791e-05    & 4.86     & 6.9407e-03   & 1.06  & 4.1336e-01  & 0.064  \\
1/40 &  3.5111e-08  &  4.90     & 4.2020e-07    & 4.93     & 3.2787e-03   & 1.08  & 3.9053e-01  & 0.082  \\
1/80 &  1.1391e-09  &  4.95     & 1.3464e-08    & 4.96     & 1.5392e-03   & 1.09  & 3.6666e-01  & 0.091  \\
1/160&  3.6272e-11  &  4.97     & 4.2604e-10    & 4.98    & 7.2029e-04   & 1.10  & 3.4318e-01  & 0.096  \\ \hline
\end{tabular*}\label{example:3}%\vspace{-15pt}
\end{center}
\end{table}

Table \ref{example:3} numerically verifies Theorem \ref{theorem4.13}, i.e., the scheme (\ref{4.12}) can keep the high convergent order when 
the regularity requirements of the performed functions are relaxed; but the scheme  (\ref{3.8}) fails.

\section{Conclusions}
When studying the anomalous diffusion, CTRW is the most widely used model. However, if the boundary conditions and external fields are needed to consider, the equations are more convenient to include these quantities. Assuming the probability density functions (PDFs) of the waiting time and jump lengths in CTRW model are independent, from CTRW model  we can derive the corresponding fractional partial differential equations (PDEs). On the other cases, when the PDFs of the CTRW model are coupled in some way, the derived PDEs usually have a fractional substantial derivative/integral. Nowadays, it seems that there are less mathematical works for this kind of operators. This paper detailedly discusses the properties of fractional substantial calculus, and provide a series of high order discretization schemes, which does well preparation for numerically solving PDEs with fractional substantial calculus.

%Because of the fractional substantial derivative, which represents important nonlocal couplings in time and space.
%So thay play more and more important in  physical application.
%In this paper, we derive some important properties of the fractional substantial  calculus.
%And using  Fourier transform methods and fractional linear multistep methods to analyse the  convergence with the global truncation error $\mathcal{O}(h^p)$ $(p=1,2,3,4,5)$ of  fractional substantial calculus, respectively.

\section*{Acknowledgements} We thanks Eli Barkai for the fruitful discussions and letting us know the urgency to solve the PDE with fractional substantial derivative in physical community.

\end{document}